\documentclass[reqno,12pt]{amsart}
\usepackage{amsmath,amssymb, amsthm}
\usepackage{hyperref}
 \newtheorem{theorem}{Theorem}[section]
\newtheorem{proposition}{Proposition}[section]

\newtheorem{lemma}{Lemma}[section]
 
\usepackage{hyperref}
\def\eps{\varepsilon}
\usepackage{color}

\begin{document}
\title{Liouville-type theorems for polyharmonic H\'enon-Lane-Emden system} 

\author{Quoc Hung PHAN} 

\address{{Institute of Research and Development,
Duy Tan University, Da Nang, Vietnam}}

\email{phanquochung@dtu.edu.vn}

\begin{abstract}
We study Liouville-type theorem for  polyharmonic H\'enon-Lane-Emden system $(-\Delta)^mu=|x|^av^p,\; (-\Delta)^mv=|x|^bu^q$  when $m,p,q\geq 1, pq\ne 1$,  and $a,b\geq 0$. It is a natural conjecture that the nonexistence of positive solutions should be true if and only if $(N+a)/(p+1)+$ $(N+b)/(q+1)>N-2m$.  It is shown by Fazly\cite{Faz14} that the conjecture holds for radial solutions in all dimensions and for classical solutions in dimension $N\leq 2m+1$. We here give some partial results in dimension $N\geq 2m+2$.
\end{abstract}
\maketitle

\section{Introduction}
In this paper, we study the following polyharmonic H\'enon-Lane-Emden  system 
\begin{align}\label{1}
\begin{cases}(-\Delta)^mu=|x|^av^p,\quad x\in\Omega,\\
(-\Delta)^mv=|x|^bu^q,\quad x\in\Omega,
\end{cases}
\end{align}
where $a,b\geq 0$, $m,p,q\geq 1$, $pq\ne 1$  and $\Omega$ is a domain of ${\mathbb R}^N$. 
Our primary interest is the Liouville property -- i.e. the nonexistence of solution 
in the entire space $\Omega={\mathbb R}^N$. Throughout of this paper, we restrict ourselves to the case of nonnegative classical solutions.

We first recall the counterpart (\ref{1}) when $m=1$, the so-called H\'enon-Lane-Emden system 
\begin{align}\label{1a}
\begin{cases}
-\Delta u=|x|^av^p,\quad x\in \Omega,\\
-\Delta v=|x|^bu^q,\quad x\in \Omega,
\end{cases}
\end{align}
which has been extensively studied by many authors. The H\'enon-Lane-Emden conjecture states that there is no positive classical solution to (\ref{1a})  in $\Omega={\mathbb R}^N$ if and only if 
\begin{align}\label{hyper}
\frac{N+a}{p+1}+\frac{N+b}{q+1}>N-2.
\end{align}
For the case $a=b=0$, this conjecture  is known to be true for radial solutions in all dimensions \cite{Mit96, SZ98}. For non-radial solutions, in dimension $N\leq 2$, the conjecture is a consequence of a result of Mitidieri and Pohozaev \cite{MP01}. In dimension $N=3$, it was proved by Serrin and Zou \cite{SZ96} under the additional assumption that $(u,v)$ has at most polynomial growth at $\infty$. This assumption was then removed by Polacik, Quittner and Souplet \cite{PQS07} and hence the conjecture is true for  $N=3$. Recently, the conjecture was proved in dimension $N=4$ by Souplet \cite{Sou09}. Some  partial results were also established for $N\geq 5$,  see \cite{Sou09, BM02, Lin98, CL09}.
For the case $a\ne 0$ or $b\ne 0$, the conjecture was proved for the case of radial solutions in all dimensions (see \cite{BVG10}). More recently, it was proved  in dimension $N\leq 3$ for the class of bounded solutions (see \cite{FG, Pha12}). In higher dimensions, the conjecture is still open, only some partial results were obtained (see e.g.\cite{Pha12} and references therein). 

Concerning the scalar counterpart of (\ref{1}), $(-\Delta)^m u=|x|^au^p$, the Liouville type result is  completely solved when $a=0$ (see \cite{Lin98,WX99}). The critical exponent in this case is $p_S(m)=\frac{N+2m}{N-2m}$. When $a\ne 0$, the Liouville type result was proved in dimensions $N\leq 2m+1$ -- see \cite{PhS} for $m=1$, Cowan \cite{Cow14} for $m=2$, and Fazly \cite{Faz14} for any $m\geq 1$.

We now return to general system (\ref{1}), it has been conjectured that the critical hyperbola is the following
\begin{align}\label{hyperbola1}
\frac{N+a}{p+1}+\frac{N+b}{q+1}=N-2m.
\end{align}

\noindent{\bf Conjecture A.} {\it Suppose that $(p, q)$ is below the critical hyperbola (\ref{hyperbola1}), i.e.
\begin{align}\label{hyperbola}
\frac{N+a}{p+1}+\frac{N+b}{q+1}>N-2m,
\end{align}
 then there is no positive solution in ${\mathbb R}^N$ of system (\ref{1a}). }

So far, this conjecture was proved for the class of radial solutions in any dimension (see \cite{Faz14, LGZ06}). For nonradial solutions, this conjecture is true for dimension $N\leq 2m$. In fact, the Liouville-type theorem for super-solutions was completely proved under a stronger  assumption (e.g \cite{MP01}), namely
$$\max\Big\{\frac{2m(p+1)}{pq-1}+\frac{a+bp}{pq-1}, \frac{2m(q+1)}{pq-1}+\frac{b+aq}{pq-1}\Big\}\geq N-2m.$$
More recently,  Fazly in \cite{Faz14}  has proved this conjecture for general case in dimension $N=2m+1$ in the class of bounded solutions. And independently, for the special case $a=b=0$, the authors in \cite{AYZ14} have proved this conjecture in dimension $N=2m+1$ and $N=2m+2$. The proofs in \cite{Faz14, AYZ14} are essential refinements of those of Souplet \cite{Sou09}. 

In this paper, we establish the Liouville-type theorems for system (\ref{1a}) in higher dimensions. Our results are the following 

\begin{theorem}\label{th1}
Let $(u,v)$ be nonnegative bounded solution of (\ref{1}) in $\mathbb R^N$. Assume (\ref{hyperbola}) and
\begin{align}\label{condi1}\max\Big\{\frac{2m(p+1)}{pq-1}, \frac{2m(q+1)}{pq-1}\Big\}\geq N-2m-1.
\end{align}
Then $u\equiv v\equiv 0$.
\end{theorem}

\begin{theorem}\label{th2}
Let $(u,v)$ be nonnegative solution of (\ref{1}) in $\mathbb R^N$.  Assume (\ref{condi1}) and 
\begin{align}\label{condi2}
\frac{N}{p+1}+\frac{N}{q+1}>N-2m.
\end{align}
Then $u\equiv v\equiv 0$.
\end{theorem}
The proof of Theorem~\ref{th1} employs the technique of Souplet \cite{Sou09}, which is based on a combination of Rellich-Pohozaev identity, Sobolev and interpolation inequality on $S^{N-1}$ and feedback and measure arguments. However, we point out that significant additional difficulties arise in our case. For instance, the very technical measure and feedback arguments become even more complicated when we work on high polyharmonic operator, and the presence of different homogeneities of  weight functions and  makes the problem much more delicate. In the difference from the case in \cite{Faz14}, the failure of the Sobolev embedding $W^{2m,1+\varepsilon}\subset L^\infty$ on $S^{N-1}$ for dimensions $N\geq 2m+2$ makes the proof become significantly more difficult. Besides that, we can not derive the comparison  between components $u$ and $v$ due to the lack of maximum principle when dealing with high order operator, we follow the idea in \cite{AYZ14} to overcome this difficulty. 

Theorem~\ref{th1} is still true for polynomially bounded solutions, i.e. if $u(x)+ v(x)\leq C|x|^s$ for $x$ large, with some $s>0$. This follows from easy modifications of the proof. Let us recall that Liouville type theorems for bounded solutions are usually sufficient for applications such as a priori estimates and universal bounds, obtained by rescaling arguments (see \cite{GS81b, PQS07}).

The proof of Theorem~\ref{th2} is based on a combination of Rellich-Pohozaev identity and decay estimates of solutions which is derived from Liouville-type results for polyharmonic Lane-Emden system. Interestingly, no boundedness assumption on solutions is required.

The rest of paper is organized as follows. In Section 2, we recall some functional inequalities, Rellich-Pohozaev identity. Section 3 is devoted to the proof of Theorem \ref{th1}. The proof of Theorem \ref{th2} is  given in Section 4.

\section{Preliminaries}
For $R>0$, we set $B_R=\{x\in {\mathbb R}^N;\ |x|<R\}$.
We shall use spherical coordinates $r=|x|$, $\theta=x/|x| \in S^{N-1}$ and write $u=u(r,\theta)$. 
The surface measures on $S^{N-1}$ and on the sphere
$\{x\in {\mathbb R}^N;\ |x|=R\}$, $R>0$, will be denoted respectively by $d\theta$ and by $d\sigma_R$.
For given function $w=w(\theta)$ on $S^{N-1}$ and $1\leq k\leq\infty$, we set
$\|w\|_k=\|w\|_{L^k(S^{N-1})}$. 
When no confusion is likely, we shall denote
$\|u\|_k=\|u(r,\cdot)\|_k$.

\subsection{Some functional inequalities}
We recall some following fundamental interpolation inequalities and elliptic estimates.
\begin{lemma}[Sobolev inequalities on $S^{N-1}$]\label{1b}
Let $N\geq 2, j\geq 1$ is integer and $1<k<\lambda \leq\infty$, $k\ne (N-1)/j$. For $w=w(\theta)\in W^{j,k}(S^{N-1})$, we have
\begin{align*}
\|w\|_\lambda \leq C(\|D^j_\theta w\|_k+\|w\|_1)
\end{align*}
where 
\begin{align*}
\begin{cases}
&\frac{1}{k}-\frac{1}{\lambda}=\frac{j}{N-1}, \text{ if } k<(N-1)/j,\\
&\lambda=\infty \quad\quad\quad\, \text{ if } k>(N-1)/j,
 \end{cases}
\end{align*}
and $C=C(j,k,N)>0$.
 \end{lemma}
The next two lemmas follow from the standard estimates for $R=1$ and an obvious dilation argument. The proof of Lemma~\ref{lem2a} makes use of standard elliptic $L^p$- estimates for second order elliptic equations and interpolation inequalities.
\begin{lemma}[Elliptic $L^p$- estimates on an annulus]\label{lem2a}
Let $1<k<\infty$. For $R>0$ and $z=z(x)\in W^{2m,k}(B_{2R}\setminus B_{R/4})$, we have
\begin{align*}
\int\limits_{B_R\setminus B_{R/2}}|D^{2m}_xz|^kdx\leq C\bigg(\int\limits_{B_{2R}\setminus B_{R/4}}|\Delta^m z|^kdx+R^{-2mk} \int\limits_{B_{2R}\setminus B_{R/4}}|z|^kdx\bigg),
\end{align*}
with $C=C(m,N,k)>0$.
\end{lemma}

%\begin{lemma}[An interpolation inequality on an annulus]\label{lem3a}
%For $R>0$ and $z=z(x)\in W^{2m,1}(B_{2R}\setminus B_{R/4})$, we have
%\begin{align*}
%\int\limits_{B_R\setminus B_{R/2}}|D^i_xz|dx\leq CR^{2m-i}%\int\limits_{B_{2R}\setminus B_{R/4}}|\Delta^{m} z|dx+ CR^{-i}\int\limits_{B_{2R}\setminus B_{R/4}}|z|dx, i=1,2,..., 2m-1
%\end{align*}
%with $C=C(N,m,i)>0$.
%\end{lemma}

\subsection{Basic estimates and Rellich-Pohozaev identity}

For the sake of simplicity, we denote by 
\begin{align*}
&k=\frac{p+1}{p},\quad h=\frac{q+1}{q},\\
&u_i=\Delta^i u, v_i=\Delta^i v,\quad i=0,1,..., m,\\ 
&F(R)=\int_{B_R}(|x|^bu^{q+1}+|x|^av^{q+1})dx.
\end{align*}
We set
\begin{align*}
\alpha=\frac{2m(p+1)}{pq-1},\quad \beta=\frac{2m(q+1)}{pq-1}.
\end{align*}
Then the condition (\ref{hyperbola}) in Conjecture A is equivalent to 
\begin{align}\label{hyperbola2}
\alpha+\beta+2m-N+\frac{b}{2m}\alpha+\frac{a}{2m}\beta>0.
\end{align}
 We have the following basic integral estimates for solutions of (\ref{1}). 

\begin{lemma}[\cite{MP01, Faz14}]\label{lem4a}
Let $p,q\geq 1$, $pq\ne 1$, $a,b\geq 0$, and $(u,v)$ be a positive solution of (\ref{1}) in $\Omega={\mathbb R}^N$. Then for any $R>0$,  there hold
\begin{align}\label{01}
&\int\limits_{B_R\setminus B_{R/2}}|x|^av^p\,dx\le C R^{N-2m-\alpha-\frac{a+bp}{pq-1}}, \;\;\; 
   \int\limits_{B_R\setminus B_{R/2}}|x|^bu^qdx\le C R^{N-2m-\beta-\frac{b+aq}{pq-1}},
\end{align} 
with $C=C(N,m,p,q,a,b)>0$.
\end{lemma}

By applying Lemma~\ref{lem2a}--\ref{lem4a} and the boundedness of $(u,v)$, we obtain the following estimates on the derivatives of $u$ and $v$.

\begin{lemma}\label{lem}
Let $p,q\geq1$, $pq\ne 1$, $a,b\geq 0$, and $(u,v)$ be a bounded positive solution of (\ref{1}), there exists a constant $C$ independent of $R$ such that 
\begin{align}
&\begin{cases} \int\limits_{B_R\setminus B_{R/2}}|D_x^{2m} u|^{1+\varepsilon}dx\le C R^{N-2m-\alpha-\frac{a+bp}{pq-1}+a\varepsilon} , \\
   \int\limits_{B_R\setminus B_{R/2}}|D^{2m}_x v|^{1+\varepsilon}dx\le C R^{N-2m-\beta-\frac{b+aq}{pq-1}+b\varepsilon},\end{cases}\label{e1}\\
 &\begin{cases}\int\limits_{B_R\setminus B_{R/2}}|D_x^{2m} u|^{k}dx\le C R^{\frac{a}{p}}F(2R) , \\
   \int\limits_{B_R\setminus B_{R/2}}|D^{2m}_x v|^{h}dx\le C R^{\frac{b}{p}} F(2R), \end{cases}\label{e2}\\
 &\begin{cases}\int\limits_{B_R\setminus B_{R/2}}|u_i|dx\le C R^{N-2i-\alpha-\frac{a+bp}{pq-1}},\quad i=0,1,...,m, \\
\int\limits_{B_R\setminus B_{R/2}}|v_i|dx\le C R^{N-2i-\beta-\frac{b+aq}{pq-1}},\quad i=0,1,...,m, \end{cases}\label{e3}\\
 &\begin{cases}\int\limits_{B_R\setminus B_{R/2}}|D_xu_j|dx\le C R^{N-2j-1-\alpha-\frac{a+bp}{pq-1}},j=0,..., m-1,\\
\int\limits_{B_R\setminus B_{R/2}}|D_xv_j|dx\le C R^{N-2j-1-\beta-\frac{b+aq}{pq-1}}, j=0,..., m-1. \end{cases}\label{e4}
  \end{align}
\end{lemma}

\begin{proof}
By using Lemmas \ref{lem2a}, \ref{lem4a}, and the boundedness of $(u,v)$, we have
\begin{align*}
 \int\limits_{B_R\setminus B_{R/2}}|D_x^{2m}u|^{1+\eps}dx&\leq 
C\int\limits_{B_{2R}\setminus B_{R/4}}|\Delta^m u|^{1+\eps}dx
+CR^{-2m(1+\eps)} \int\limits_{B_{2R}\setminus B_{R/4}}u^{1+\eps}dx\\
&\leq 
C\int\limits_{B_{2R}\setminus B_{R/4}}|x|^{a+a\eps}v^{p+p\eps}dx
+CR^{-2m(1+\eps)}\int\limits_{B_{2R}\setminus B_{R/4}}u^{1+\eps}dx\\
&\leq 
CR^{a\eps}\int\limits_{B_{2R}\setminus B_{R/4}}|x|^av^pdx
+CR^{-2m(1+\eps)}\int\limits_{B_{2R}\setminus B_{R/4}}u\,dx\\
&\leq CR^{N-2m-\alpha-\frac{a+bp}{pq-1}+ a\eps}+CR^{N-2m-\alpha-\frac{a+bp}{pq-1}-2m\eps} \\
&\leq CR^{N-2-\alpha-\frac{a+bp}{pq-1}+ a\eps}.
\end{align*}
The second inequality of (\ref{e1}) holds by similar calculation. Next,
\begin{align*}
\int\limits_{B_R\setminus B_{R/2}}|D_x^{2m}u|^kdx&\leq C\bigg(\int\limits_{B_{2R}\setminus B_{R/4}}|\Delta^m u|^kdx+ R^{-2mk}\int\limits_{B_{2R}\setminus B_{R/4}}u^kdx\bigg)\\
&= C\bigg(\int\limits_{B_{2R}\setminus B_{R/4}}|x|^{ka}v^{p+1}dx+ R^{-2mk}\int\limits_{B_{2R}\setminus B_{R/4}}u^kdx\bigg)\\
&\leq  C\bigg(R^{a/p}F(2R)+ R^{-2mk}\int\limits_{B_{2R}\setminus B_{R/4}}u^kdx\bigg).
\end{align*}
By H\"older inequality, for $R>1$, we have
\begin{align*}
A_1&=R^{-2mk}\int\limits_{B_{2R}\setminus B_{R/4}}u^kdx\\
&\leq CR^{-2mk}R^{N(pq-1)/p(q+1)}\bigg(\int\limits_{B_{2R}\setminus B_{R/4}}u^{q+1}dx\bigg)^{(p+1)/p(q+1)}\\
&\leq C R^{\eta_1/p}F(2R),
\end{align*}
with $\eta_1=-2m(p+1)+N(pq-1)/(q+1)- b(p+1)/(q+1)$, where we used $(p+1)/p(q+1)<1$, along with 
\begin{align*}
F(R)\geq F(1)>0, \quad R>1.
\end{align*}
We have to show that $\eta_1<a$. Indeed,
\begin{align*}
a-\eta_1 &=2m(p+1)-N\frac{pq-1}{q+1}+ b\frac{p+1}{q+1}+a\\
&=2m(p+1)-N\frac{pq-1}{q+1}+ b\frac{p(q+1)-(pq-1)}{q+1}+a\\
&=\frac{2m}{\beta}\left((p+1)\beta-N+ \frac{b}{2m}p\beta-b+\frac{a}{2m}\beta\right)\\
&=\frac{2m}{\beta}\left(2m+\alpha+\beta-N+ \frac{b}{2m}\alpha+\frac{a}{2m}\beta\right)>0.
\end{align*}
The second inequality of (\ref{e2}) holds by similar calculation. 

Inequalities (\ref{e3}) follow from Lemmas \ref{lem4a} and regularity estimate for Laplace equation and interpolation inequalities. For the proof of (\ref{e4}), it follow from \cite[Lemma 4.3]{PhS} that
\begin{align*}
\int\limits_{B_R\setminus B_{R/2}}|D_xu_j|dx\leq CR\int\limits_{B_{2R}\setminus B_{R/4}}|u_{j+1}|dx+ CR^{-1}\int\limits_{B_{2R}\setminus B_{R/4}}|u_j|dx. 
\end{align*}
This combined with (\ref{e3}) yields desired estimates.
\end{proof}

The following Rellich-Pohozaev identity plays a key role in the proof of Theorems \ref{th1} and \ref{th2}. The proof of this identity is standard and we refer to \cite{AYZ14, Faz14} for details.  
\begin{lemma}[Rellich-Pohozaev identity]\label{pohozaev}
Let $a_1, a_2\in {\mathbb R}$ satisfy $a_1+a_2=N-2m$ and $(u,v)$ solution of (\ref{1}), there holds 
\begin{align*}
\bigg(&\frac{N+a}{p+1}-a_1\bigg)\int\limits_{B_R}|x|^av^{p+1}dx+\bigg(\frac{N+b}{q+1}-a_2\bigg)\int\limits_{B_R}|x|^bu^{q+1}dx\\
=& \frac{1}{q+1}R^{1+b}\int\limits_{|x|=R}u^{q+1}d\sigma_R+\frac{1}{p+1}R^{1+a} \int\limits_{|x|=R}v^{p+1}d\sigma_R\\
&-(-1)^m\sum_{i=0}^{m-1}2R\int\limits_{|x|=R}\frac{\partial \Delta^i u}{\partial n}.\frac{\partial \Delta^{m-i-1} v}{\partial n}d\sigma_R\\
&-(-1)^m\sum_{i=0}^{m-1}R\int\limits_{|x|=R} (\nabla \Delta^i u,\nabla\Delta^{m-i-1} v)d\sigma_R\\
&-(-1)^m\sum_{i=0}^{m-2}R\int\limits_{|x|=R}  \Delta^{i+1} u.\Delta^{m-i-1} vd\sigma_R\\
&+(-1)^m\sum_{i=0}^{m-1}(2m-2i-2+a_1)\int\limits_{|x|=R} \frac{\partial \Delta^i u}{\partial n}.\Delta^{m-i-1} vd\sigma_R\\
&+(-1)^m\sum_{i=0}^{m-1}(2i+a_2)\int\limits_{|x|=R}\Delta^{i} u. \frac{\partial \Delta^{m-i-1} v}{\partial n}d\sigma_R.
\end{align*}
\end{lemma}

\vskip 0.2cm

\section{Proof of Theorem~\ref{th1} }

{\it Proof of Theorem~\ref{th1}.} Since the theorem was proved in dimension $N\leq 2m+1$ (see \cite{Faz14}), we only prove the Theorem in dimension $N\geq 2m+2$. Without loss of generality, we assume $p\geq q$. Then we may assume in addition that 
\begin{align}\label{condp}
p\geq \frac{N+2m}{N-2m}.
\end{align}
In fact, if $q\le p< \frac{N+2m}{N-2m}$, then we may apply Theorem \ref{th2} (which will be proved independently of Theorem \ref{th1} in Section 5). The proof is based on the idea of Souplet, which consists of 5 steps. We repeat these steps in detail for completeness and because of the additional technicalities coming from the coefficient $|x|^a, |x|^b$ and from polyharmonic operator. Suppose that there exists a positive solution $(u, v)$ of (\ref{1}) in ${\mathbb R}^N$.

\noindent{\bf Step 1:} {\it Preparations. } Let us choose $a_1, a_2$ such that $a_1+a_2=N-2m$ and 
\begin{equation}\label{choicea1a2}
\frac{N+a}{p+1}>a_1, \; \frac{N+b}{q+1}>a_2
\end{equation}
and set $F(R)=\int_{B_R}(|x|^bu^{q+1}+|x|^av^{p+1})dx$.
By using the Rellich-Pohozaev identity (Lemma \ref{pohozaev}) and noting $|x|^av^{p+1}=(-\Delta)^mu.v,\; |x|^bu^{q+1}=(-\Delta)^mv.u$, we have
\begin{align*}
F(R)\leq C\big(G_1(R)+G_2(R)\big),
\end{align*}
where 
\begin{align}\label{G1}
 &G_1(R)=R^{N}\sum\limits_{i=0}^{m}\int\limits_{S^{N-1}}|u_i|. |v_{m-i}|d\theta,\\
&G_2(R)=R^N\sum\limits_{j=0}^{m-1}\int\limits_{S^{N-1}}\left(|D_xu_{j}|+R^{-1}|u_{j}|\right)\left(|D_xv_{m-j-1}|+R^{-1}|v_{m-j-1}| \right)d\theta.\label{G2}
\end{align}
We denote by 
\begin{align*}
k=\frac{p+1}{p}, \quad h=\frac{q+1}{q}.
\end{align*}
\noindent{\bf Step 2:} {\it Estimation of $G_1(R)$.}
Fix $i\in \{0,1,...,m\}$, we consider following two cases

{\it Case $\frac{1}{h}>\frac{2i}{N-1}$}: We set
\begin{align*}
&\frac{1}{\gamma_{1i}}=1-\frac{2m-2i}{N-1}, \quad \frac{1}{\delta_{1i}}=\frac{1}{k}-\frac{2m-2i}{N-1},\\
&\frac{1}{\lambda_{1i}}=1-\frac{2i}{N-1}, \quad \frac{1}{\rho_{1i}}=\frac{1}{h}-\frac{2i}{N-1}.
\end{align*}
Let $z_i\in (1,+\infty)$ such that
\begin{align*}
\frac{1}{z_i}=\min\Big\{1-\frac{2m-2i}{N-1}, \frac{1}{q+1}+\frac{2i}{N-1}\Big\},
\end{align*}
then we have 
\begin{align*}
&\frac{1}{\delta_{1i}}\leq \frac{1}{z_i}\leq \frac{1}{\gamma_{1i}},\\
&\frac{1}{\rho_{1i}}\leq 1-\frac{1}{z_i}=\frac{1}{z_i'}\leq \frac{1}{\lambda_{1i}}.
\end{align*}
It follows from H\"older inequality that
\begin{align*}
\int\limits_{S^{N-1}}|u_i|.|v_{m-i}|d\theta \leq C\|u_i\|_{z_i}\|v_{m-i}\|_{z_i'}
\leq C \|u_i\|_{\gamma_{1i}}^{\nu_{1i}}\|u_i\|_{\delta_{1i}}^{1-\nu_{1i}}\|v_{m-i}\|_{\lambda_{1i}}^{\tau_{1i}}\|v_{m-i}\|_{\rho_{1i}}^{1-\tau_{1i}},
\end{align*}
where 
\begin{align}\label{nu}
\nu_{1i}=1-\frac{\frac{1}{\gamma_{1i}}-\frac{1}{z_i}}{\frac{1}{\gamma_{1i}}-\frac{1}{\delta_{1i}}}, \qquad \tau_{1i}=1-\frac{\frac{1}{\lambda_{1i}}-\frac{1}{z'_i}}{\frac{1}{\lambda_{1i}}-\frac{1}{\rho_{1i}}}.
\end{align}
Applying Lemma~\ref{1b}, we have 
\begin{align*}
\|u_i\|_{\gamma_{1i}}&\leq C (\|D^{2m-2i}_\theta u_i\|_{1+\varepsilon}+\|u_i\|_1)\leq C (R^{2m-2i}\|D^{2m-2i}_x u_i\|_{1+\varepsilon}+\|u_i\|_1)\\
&\leq CR^{2m-2i}(\|D^{2m}_x u\|_{1+\varepsilon}+R^{-2m+2i}\|u_i\|_1),\\
\|u_i\|_{\delta_{1i}}&\leq C (\|D^{2m-2i}_\theta u_i\|_{k}+\|u_i\|_1)\leq C (R^{2m-2i}\|D^{2m-2i}_x u_i\|_{k}+\|u_i\|_1)\\
&\leq CR^{2m-2i}(\|D^{2m}_x u\|_{k}+R^{-2m+2i}\|u_i\|_1),\\
\|v_{m-i}\|_{\lambda_{1i}}&\leq C (\|D^{2i}_\theta v_{m-i}\|_{1+\varepsilon}+\|v_{m-i}\|_1)\leq C (R^{2i}\|D^{2i}_x v_{m-i}\|_{1+\varepsilon}+\|v_{m-i}\|_1)\\
&\leq CR^{2i}(\|D^{2m}_x v\|_{1+\varepsilon}+R^{-2i}\|v_{m-i}\|_1),\\
\|v_{m-i}\|_{\rho_{1i}}&\leq C (\|D^{2i}_\theta v_{m-i}\|_{h}+\|v_{m-i}\|_1)\leq C (R^{2i}\|D^{2i}_x v_{m-i}\|_{h}+\|v_{m-i}\|_1)\\
&\leq CR^{2i}(\|D^{2m}_x v\|_{h}+R^{-2i}\|v_{m-i}\|_1).
\end{align*}
Consequently, 
\begin{align}
\int\limits_{S^{N-1}}|u_i|.|v_{m-i}|d\theta \leq C R^{2m}&(\|D^{2m}_x u\|_{1+\varepsilon}+R^{-2m+2i}\|u_i\|_1)^{\nu_{1i}}\notag\\
&.(\|D^{2m}_x u\|_{k}+R^{-2m+2i}\|u_i\|_1)^{1-\nu_{1i}}\notag\\
&.(\|D^{2m}_x v\|_{1+\varepsilon}+R^{-2i}\|v_{m-i}\|_1)^{\tau_{1i}}\notag\\
&.(\|D^{2m}_x v\|_{h}+R^{-2i}\|v_{m-i}\|_1)^{1-\tau_{1i}}.\label{hol}
\end{align}

{\it Case $\frac{1}{h}\leq\frac{2i}{N-1}$}: We take $\tau_{1i}=1$ then inequality (\ref{hol}) still holds.

In two cases, we always have 
\begin{align*}
G_1(R) \leq C \sum\limits_{i=0}^{m} R^{N+2m}&(\|D^{2m}_x u\|_{1+\varepsilon}+R^{-2m+2i}\|u_i\|_1)^{\nu_{1i}}\\
&.(\|D^{2m}_x u\|_{k}+R^{-2m+2i}\|u_i\|_1)^{1-\nu_{1i}}\notag\\
&. (\|D^{2m}_x v\|_{1+\varepsilon}+R^{-2i}\|v_{m-i}\|_1)^{\tau_{1i}}\\
&.(\|D^{2m}_x v\|_{h}+R^{-2i}\|v_{m-i}\|_1)^{1-\tau_{1i}}
\end{align*}
where $\nu_{1i}, \tau_{1i}$ are defined in (\ref{nu}), in which $\tau_{1i}=1$  if $\frac{1}{h}\leq \frac{2i}{N-1}$.

\noindent{\bf Step 3:} {\it Estimation of $G_2(R)$.}
Fix $j\in \{0,1,...,m-1\}$, we consider following two cases

{\it Case $\frac{1}{h}>\frac{2j+1}{N-1}$}: We set
\begin{align*}
&\frac{1}{\gamma_{2j}}=1-\frac{2m-2j-1}{N-1}, \quad \frac{1}{\delta_{2j}}=\frac{1}{k}-\frac{2m-2j-1}{N-1},\\
&\frac{1}{\lambda_{2j}}=1-\frac{2j+1}{N-1}, \quad \frac{1}{\delta_{2j}}=\frac{1}{h}-\frac{2j+1}{N-1}.
\end{align*}
Let $t_j\in (1,+\infty)$ such that
\begin{align*}
\frac{1}{t_j}=\min\Big\{1-\frac{2m-2j-1}{N-1}, \frac{1}{q+1}+\frac{2j+1}{N-1}\Big\},
\end{align*}
then we have 
\begin{align*}
&\frac{1}{\delta_{2j}}\leq \frac{1}{t_j}\leq \frac{1}{\gamma_{2j}},\\
&\frac{1}{\rho_{2j}}\leq 1-\frac{1}{t_j}=\frac{1}{t_j'}\leq \frac{1}{\lambda_{2j}}.
\end{align*}
It follows from Lemma~\ref{1b} that
\begin{align*}
&R^{-1}\|u_{j}\|_{t_j}\leq CR^{-1}(\|D_\theta u_{j}\|_{t_j}+\|u_{j}\|_{1})\leq C(\|D_x u_{j}\|_{t_j}+R^{-1}\|u_{j}\|_{1}),\\
&R^{-1}\|v_{m-j-1}\|_{t_j'}\leq CR^{-1}(\|D_\theta v_{m-j-1}\|_{t_j'}+\|v_{m-j-1}\|_1)\\
&\qquad\qquad\leq C(\|D_x v_{m-j-1}\|_{t_j'}+R^{-1}\|v_{m-j-1}\|_1).
\end{align*}
This combined with H\"older inequality yields
\begin{align*}
\int\limits_{S^{N-1}}&\left(|D_xu_{j}|+R^{-1}|u_{j}|\right)\left(|D_xv_{m-j-1}|+R^{-1}|v_{m-j-1}| \right)d\theta\\
&\leq (\|D_xu_{j}\|_{t_j}+R^{-1}\|u_{j}\|_{t_j}).(\|D_xv_{m-j-1}\|_{t_j'}+R^{-1}\|v_{m-j-1}\|_{t_j'})\\
&\leq (\|D_xu_{j}\|_{t_j}+R^{-1}\|u_{j}\|_{1}).(\|D_xv_{m-j-1}\|_{t_j'}+R^{-1}\|v_{m-j-1}\|_{1})\\
&\leq (\|D_xu_{j}\|_{\gamma_{2j}}+R^{-1}\|u_{j}\|_{1})^{\nu_{2j}}.(\|D_xu_{j}\|_{\delta_{2j}}+R^{-1}\|u_{j}\|_{1})^{1-\nu_{2j}}\\
&\;\quad.(\|D_xv_{m-j-1}\|_{\lambda_{2j}}+R^{-1}\|v_{m-j-1}\|_{1})^{\tau_{2j}}\\
&\;\quad.(\|D_xv_{m-j-1}\|_{\rho_{2j}}+R^{-1}\|v_{m-j-1}\|_{1})^{1-\tau_{2j}}
\end{align*}
where
\begin{align*}
\nu_{2j}=1-\frac{\frac{1}{\gamma_{2j}}-\frac{1}{t_j}}{\frac{1}{\gamma_{2j}}-\frac{1}{\delta_{2j}}}, \qquad \tau_{2j}=1-\frac{\frac{1}{\lambda_{2j}}-\frac{1}{t'_j}}{\frac{1}{\lambda_{2j}}-\frac{1}{\rho_{2j}}}.
\end{align*}
On the other hand, applying Lemma~\ref{1b}, we have 
\begin{align*}
\|D_xu_{j}\|_{\gamma_{2j}}&\leq C (\|D^{2m-2j-1}_\theta D_xu_{j}\|_{1+\varepsilon}+\|D_xu_{j}\|_1)\\
&\leq C (R^{2m-2j-1}\|D^{2m-2j-1}_x u_j\|_{1+\varepsilon}+\|D_xu_j\|_1)\\
&\leq CR^{2m-2j-1}(\|D^{2m}_x u\|_{1+\varepsilon}+R^{-2m+2j+1}\|D_xu_j\|_1),\\
\|D_xu_j\|_{\delta_{2j}}&\leq C (\|D^{2m-2j-1}_\theta D_xu_j\|_{k}+\|D_xu_j\|_1)\\
&\leq C (R^{2m-2j-1}\|D^{2m-2j-1}_x D_xu_j\|_{k}+\|D_xu_j\|_1)\\
&\leq CR^{2m-2j-1}(\|D^{2m}_x u\|_{k}+R^{-2m+2j+1}\|D_xu_j\|_1),\\
\|D_xv_{m-j-1}\|_{\lambda_{2j}}&\leq C (\|D^{2j+1}_\theta D_xv_{m-j-1}\|_{1+\varepsilon}+\|D_xv_{m-j-1}\|_1)\\
&\leq C (R^{2j+1}\|D^{2j+1}_x D_xv_{m-j-1}\|_{1+\varepsilon}+\|D_xv_{m-j-1}\|_1)\\
&\leq CR^{2j+1}(\|D^{2m}_x v\|_{1+\varepsilon}+R^{-2j-1}\|D_xv_{m-j-1}\|_1),\\
\|D_xv_{m-j-1}\|_{\rho_{2j}}&\leq C (\|D^{2j+1}_\theta D_xv_{m-j-1}\|_{h}+\|D_xv_{m-j-1}\|_1)\\
&\leq C (R^{2j+1}\|D^{2j+1}_x D_xv_{m-j}\|_{h}+\|D_xv_{m-j-1}\|_1)\\
&\leq CR^{2j+1}(\|D^{2m}_x v\|_{h}+R^{-2j-1}\|D_xv_{m-j-1}\|_1).
\end{align*}
Hence,
\begin{align*}
R^N&\int\limits_{S^{N-1}}\left(|D_xu_{j}|+R^{-1}|u_{j}|\right)\left(|D_xv_{m-j-1}|+R^{-1}|v_{m-j-1}| \right)d\theta\notag\\
&\leq CR^{N+2m}(\|D^{2m-2j-1}_x u_j\|_{1+\varepsilon}+R^{-2m+2j+1}\|D_xu_j\|_1+R^{-2m+2j}\|u_{j}\|_{1})^{\nu_{2j}}\notag\\
&\qquad \quad .(\|D^{2m}_x u\|_{k}+R^{-2m+2j+1}\|D_xu_j\|_1+R^{-2m+2j}\|u_{j}\|_{1})^{1-\nu_{2j}}\notag\\
&\qquad\quad.(\|D^{2m}_x v\|_{1+\varepsilon}+R^{-2j-1}\|D_xv_{m-j-1}\|_1+R^{-2j-2}\|v_{m-j-1}\|_{1})^{\tau_{2j}}\notag\\
&\qquad\quad .(\|D^{2m}_x v\|_{h}+R^{-2j-1}\|D_xv_{m-j-1}\|_1+R^{-2j-2}\|v_{m-j-1}\|_{1})^{1-\tau_{2j}}.%\label{hol3}
\end{align*}
{\it Case $\frac{1}{h}\leq \frac{2j+1}{N-1}$}: We take $\tau_{2j}=1$ then inequality (\ref{hol}) still holds.
Therefore,
\begin{align*}
G_2(R)\!\leq\! C&R^{N+2m}{\textstyle\sum\limits_{j=0}^{m\!-\!1}}(\|D^{2m-2j-1}_x u_j\|_{1+\varepsilon}\!+\!R^{-2m+2j+1}\|D_xu_j\|_1\!+\!R^{-2m+2j}\|u_{j}\|_{1})^{\nu_{2j}}\notag\\
&\qquad \quad .(\|D^{2m}_x u\|_{k}+R^{-2m+2j+1}\|D_xu_j\|_1+R^{-2m+2j}\|u_{j}\|_{1})^{1-\nu_{2j}}\notag\\
&\qquad\quad.(\|D^{2m}_x v\|_{1+\varepsilon}+R^{-2j-1}\|D_xv_{m-j-1}\|_1+R^{-2j-2}\|v_{m-j-1}\|_{1})^{\tau_{2j}}\notag\\
&\qquad\quad .(\|D^{2m}_x v\|_{h}+R^{-2j-1}\|D_xv_{m-j-1}\|_1+R^{-2j-2}\|v_{m-j-1}\|_{1})^{1-\tau_{2j}}.%\label{hol4}
\end{align*}
\noindent{\bf Step 4:} {\it measure and feedback argument. } For a given $K>0$, let us define the sets
\begin{align*}
&\Gamma_1(R)=\{r\in (R,2R); \|D^{2m}_xu(r)\|_k^k>KR^{-N+\frac{a}{p}}F(4R)\},\\
&\Gamma_2(R)=\{r\in (R,2R); \|D^{2m}_xv(r)\|_h^h>KR^{-N+\frac{b}{q}}F(4R)\},\\
&\Gamma_3(R)=\{r\in (R,2R); \|D^{2m}_xu(r)\|_{1+\varepsilon}^{1+\varepsilon}>KR^{-2m-\alpha-\frac{a+bp}{pq-1}+a\varepsilon}\},\\
&\Gamma_4(R)=\{r\in (R,2R); \|D^{2m}_xv(r)\|_{1+\varepsilon}^{1+\varepsilon}>KR^{-2m-\beta-\frac{b+aq}{pq-1}+b\varepsilon}\},\\
&\Gamma_{5i}(R)=\{r\in (R,2R); \|u_i(r)\|_1>KR^{-2i-\alpha-\frac{a+bp}{pq-1}}\},\quad i=0,1,...,m,\\
&\Gamma_{6i}(R)=\{r\in (R,2R); \|v_i(r)\|_1>KR^{-2i-\beta-\frac{b+aq}{pq-1}}\},\quad i=0,1,...,m,\\
&\Gamma_{7j}(R)=\{r\in (R,2R); \|D_xu_j(r)\|_1>KR^{-2j-1-\alpha-\frac{a+bp}{pq-1}}\},\quad j=0,1,...,m-1,\\
&\Gamma_{8j}(R)=\{r\in (R,2R); \|D_xv_j(r)\|_1>KR^{-2j-1-\beta-\frac{b+aq}{pq-1}}\}, \quad j=0,1,...,m-1.
\end{align*}
It follows from Lemma~\ref{lem} that
\begin{align*}
CR^{a/p}F(4R)\geq \int\limits_{R}^{2R}\|D_x^{2m}u(r)\|_k^kr^{N-1}dr\geq |\Gamma_1(R)|R^{N-1}KR^{-N+\frac{a}{p}}F(4R).
\end{align*}
Consequently,
\begin{align*}
|\Gamma_1(R)|\leq \frac{C}{K}R.
\end{align*}
Similarly, we have
\begin{align*}
&|\Gamma_l(R)|\leq \frac{C}{K}R, \quad l=2,3,4,\\
&|\Gamma_{li}(R)|\leq \frac{C}{K}R \quad l=5,6; \; i=0,1,...,m,\\
&|\Gamma_{lj}(R)|\leq \frac{C}{K}R \quad l=7,8; \; j=0,1,...,m-1.
\end{align*}
By choosing $K$ large enough, we can find  
\begin{align*}
\tilde R\in(R,2R)\subset \left\{\bigcup_{l=1}^4\Gamma_l(R)\bigcup_{l=5,6; i=0,...,m} \Gamma_{li}(R)\bigcup_{l=7,8; j=0,...,m-1} \Gamma_{lj}(R)\right\}\ne \emptyset.
\end{align*}
Let us check that  %%fact that
\begin{align}
 2m+\alpha+\frac{a+bp}{pq-1}> \frac{N}{k}-\frac{a}{pk}\label{10},\\
 2m+\beta+\frac{b+aq}{pq-1}> \frac{N}{h}-\frac{b}{qh}.\label{11}
\end{align}
Indeed, by computation
\begin{align*}
M:&= (2m+\alpha)k+\frac{a+bp}{pq-1}k-N+\frac{a}{p}\\
&=p\beta k+\frac{(a+bp)(p+1)}{p(pq-1)}-N+\frac{a}{p}\\
&=\beta(p+1)+ \frac{(a+bp)}{2mp}\alpha -N+\frac{a}{p}\\
&=p\beta+\beta+\frac{a}{2mp}\alpha+ \frac{b}{2m}\alpha-N+\frac{a}{p}\\
&=\alpha+\beta+2m-N+\frac{b}{2m}\alpha+\frac{a}{2mp}(\alpha+2)\\
&=\alpha+\beta+2m-N+\frac{b}{2m}\alpha+\frac{a}{2m}\beta>0.
\end{align*}
Thus, (\ref{10}) holds. Similarly for (\ref{11}).
Hence, for $\varepsilon >0$ small enough, we have
\begin{align}
 \frac{1}{1+\varepsilon}(2m+\alpha+\frac{a+bp}{pq-1}-a\varepsilon)> \frac{N}{k}-\frac{a}{pk}\label{10a},\\
 \frac{1}{1+\varepsilon}(2m+\beta+\frac{b+aq}{pq-1}-b\varepsilon)> \frac{N}{h}-\frac{b}{qh}.\label{11a}
\end{align}
We may now control $G_1(R)$ and $G_2(R)$ as follows.
\begin{align}
 G_1(\tilde{R})\leq C R^{N+2m}\sum\limits_{i=0}^{m}&\left(R^{-\alpha-2m -\frac{a+bp}{pq-1}}+R^{(-2m-\alpha-\frac{a+bp}{pq-1}+a\varepsilon)/(1+\varepsilon)}\right)^{\nu_{1i}}\notag\\
&.\left(R^{-\beta-2m -\frac{b+aq}{pq-1}}+R^{(-2m-\beta-\frac{b+aq}{pq-1}+b\varepsilon)/(1+\varepsilon)}\right)^{\tau_{1i}}\notag\\
&.\left(R^{-\frac{N}{k}+\frac{a}{pk}}F^{1/k}(4R)+R^{-\alpha-2m -\frac{a+bq}{pq-1}}\right)^{1-\nu_{1i}}\notag\\
&.\left(R^{-\frac{N}{h}+\frac{b}{qh}}F^{1/h}(4R)+R^{-\beta-2m -\frac{b+aq}{pq-1}}\right)^{1-\tau_{1i}}.\label{31}
\end{align}
Using (\ref{10a}) - (\ref{11a}), we obtain
\begin{align}\label{30}
G_1(\tilde{R})\leq C \sum\limits_{i=0}^{m}R^{-a_{1i}(\varepsilon)}F^{b_{1i}}(4R)
\end{align}
where
\begin{align*}
a_{1i}(\varepsilon)=&-N-2m+\frac{\nu_{1i}}{1+\varepsilon}\left(\alpha+2m +\frac{a+bp}{pq-1}-a\varepsilon\right)\notag\\
&+ \frac{\tau_{1i}}{1+\varepsilon}\left(\beta+2m +\frac{b+aq}{pq-1}-b\varepsilon\right)\notag\\
&+\left(N-\frac{a}{p}\right)\frac{(1-\nu_{1i})}{k} +
\left(N-\frac{b}{q}\right)\frac{(1-\tau_{1i})}{h},%\label{15b}
\end{align*}
\begin{align*}
 b_{1i}=&\frac{1-\nu_{1i}}{k}+\frac{1-\tau_{1i}}{h}.
\end{align*}
Similarly,
\begin{align}
 G_2(\tilde{R})\leq C R^{N+2m}\sum\limits_{j=0}^{m-1}&\left(R^{-\alpha-2m -\frac{a+bp}{pq-1}}+R^{(-2m-\alpha-\frac{a+bp}{pq-1}+a\varepsilon)/(1+\varepsilon)}\right)^{\nu_{2j}}\notag\\
&.\left(R^{-\beta-2m -\frac{b+aq}{pq-1}}+R^{(-2m-\beta-\frac{b+aq}{pq-1}+b\varepsilon)/(1+\varepsilon)}\right)^{\tau_{2j}}\notag\\
&.\left(R^{-\frac{N}{k}+\frac{a}{pk}}F^{1/k}(4R)+R^{-\alpha-2m -\frac{a+bq}{pq-1}}\right)^{1-\nu_{2j}}\notag\\
&.\left(R^{-\frac{N}{h}+\frac{b}{qh}}F^{1/h}(4R)+R^{-\beta-2m -\frac{b+aq}{pq-1}}\right)^{1-\tau_{2j}}\notag\\
&\leq C \sum\limits_{j=0}^{m-1} R^{-a_{2j}(\varepsilon)}F^{b_{2j}}(4R).\label{31aa}
\end{align}
where
\begin{align*}
a_{2j}(\varepsilon)=&-N-2m+\frac{\nu_{2j}}{1+\varepsilon}\left(\alpha+2m +\frac{a+bp}{pq-1}-a\varepsilon\right)\notag\\
&+ \frac{\tau_{2j}}{1+\varepsilon}\left(\beta+2m +\frac{b+aq}{pq-1}-b\varepsilon\right)\notag\\
&+\left(N-\frac{a}{p}\right)\frac{(1-\nu_{2j})}{k} +
\left(N-\frac{b}{q}\right)\frac{(1-\tau_{2j})}{h},%\label{15b}
\end{align*}
\begin{align*}
b_{2j}=&\frac{1-\nu_{2j}}{k}+\frac{1-\tau_{2j}}{h}.
\end{align*}

We set
\begin{align*}\tilde{a}=\min\left\{a_{1i}(\varepsilon), a_{2j}(\varepsilon); \quad i=0,...,m;\; j=0,...,m-1\right\},\\
\tilde{b}=\max\{b_{1i}, b_{2j}, \quad i=0,...,m;\; j=0,...,m-1\}.
\end{align*}
 Then from (\ref{30})  and (\ref{31aa}), we obtain
\begin{align}\label{32}
 F(R)\leq C R^{-\tilde{a}}F^{\tilde{b}}(4R), R\geq 1.
\end{align}
We claim that there exist a constant $M_0>0$ and a sequence $R_l\to \infty$ such that 
\begin{align*}
 F(4R_l)\leq M_0F(R_l).
\end{align*}
Assume that the claim is false. Then, for any $M_0>0$, there exists $R_0>0$ such that $F(4R)\geq M_0F(R)$ for all $R\geq R_0$. But since $u,v$ are bounded, we have $F(R)\leq CR^{N+a+b}$. Thus 
$$M_0^lF(R_0)\leq F(4^lR_0)\leq C(4^lR_0)^{N+a+b}= CR_0^{N+a+b}4^{l(N+a+b)}, \forall l\geq 0.$$
This is a contradiction for $l$ large if we choose $M_0>4^{N+a+b}$.

Now we assume we have proved that $\tilde{a}>0$ and $\tilde{b}<1$, then from (\ref{32}) we have
$$F(R_l)\leq C R_l^{-\tilde{a}/(1-\tilde{b})}. $$
Letting $l\to \infty$, we obtain $\int_{{\mathbb R}^N}(|x|^bu^{q+1}+|x|^av^{p+1})dx=0$, hence $u\equiv v\equiv 0$: contradiction.

\medskip
\noindent{\bf Step 5:} {\it Fulfillment of the conditions $\tilde{a}>0$ and $\tilde{b}<1$}

{\it a) Verification of conditions $a_{1i}(0)>0$ and $b_{1i}<1$.} We first verify $b_{1i}<1$. 

{\it Case 1: $\frac{1}{h}\geq\frac{2m}{N-1}$.} From the definition of $z_i$, we have $\frac{1}{z_i}=\frac{1}{q+1}+\frac{2i}{N-1}$. Hence,
\begin{align*}
1-b_{1i}&=1-\frac{1}{k}.\frac{\frac{1}{\gamma_{1i}}-\frac{1}{z_i}}{\frac{1}{\gamma_{1i}}-\frac{1}{\delta_{1i}}}-\frac{1}{h}.\frac{\frac{1}{\lambda_{1i}}-\frac{1}{z'_i}}{\frac{1}{\lambda_{1i}}-\frac{1}{\rho_{1i}}}=1-p\Big(1-\frac{1}{q+1}-\frac{2m}{N-1}\Big)-\frac{q}{q+1}\\
&=\frac{pq-1}{(q+1)(N-1)}\big(\alpha-N+2m+1\big)>0.
\end{align*}

{\it Case 2: $\frac{2i}{N-1}<\frac{1}{h}<\frac{2m}{N-1}$.} From the definition of $z_i$, we have $\frac{1}{z_i}=1-\frac{2m-2i}{N-1}$. Hence,
\begin{align*}
1-b_{1i}&=1-\frac{1}{k}.\frac{\frac{1}{\gamma_{1i}}-\frac{1}{z_i}}{\frac{1}{\gamma_{1i}}-\frac{1}{\delta_{1i}}}-\frac{1}{h}.\frac{\frac{1}{\lambda_{1i}}-\frac{1}{z'_i}}{\frac{1}{\lambda_{1i}}-\frac{1}{\rho_{1i}}}\\
&=1-\frac{1}{h}.\frac{1-\frac{2m}{N-1}}{1-\frac{1}{h}}>1-\frac{1}{h}>0.
\end{align*}

{\it Case 3: $\frac{2i}{N-1}\geq\frac{1}{h}$.} Then we have $\tau_{1i}=1$. Hence,
\begin{align*}
1-b_{1i}=1-\frac{1}{k}(1-\nu_{1i})>0.
\end{align*}
Therefore $b_{1i}<1$, for any $i\in \{0,1,..., m\}$.
We next verify $a_{1i}(0)>0$. Indeed, 
\begin{align*}
a_{1i}(0)=&-N-2m+\nu_{1i}\left(\alpha+2m +\frac{a+bp}{pq-1}\right)+ \tau_{1i}\left(\beta+2m +\frac{b+aq}{pq-1}\right)\notag\\
&+\left((2m+\alpha)k+\frac{a+bp}{pq-1}k-M \right)\frac{(1-\nu_{1i})}{k} \\
&+\left((2m+\beta)h+\frac{b+aq}{pq-1}h-M\right)\frac{(1-\tau_{1i})}{h}\\
=&-N-2m+2m+\alpha+\frac{a+bp}{pq-1}+2m+\beta+\frac{b+aq}{pq-1}\\
&-M\left(\frac{1-\nu_{1i}}{k}+\frac{1-\tau_{1i}}{h}\right)\\
=& M- Mb_{1i} = M(1-b_{1i})>0.
\end{align*}

{\it b) Verification of conditions $a_{2j}(0)>0$ and $b_{2j}<1$.} We first verify $b_{2j}<1$. 

{\it Case 1: $\frac{1}{h}\geq\frac{2m}{N-1}$.} We have $\frac{1}{t_j}=\frac{1}{q+1}+\frac{2j+1}{N-1}$.
\begin{align*}
1-b_{2j}&=1-\frac{1}{k}.\frac{\frac{1}{\gamma_{2j}}-\frac{1}{t_j}}{\frac{1}{\gamma_{2j}}-\frac{1}{\delta_{2j}}}-\frac{1}{h}.\frac{\frac{1}{\lambda_{2j}}-\frac{1}{t'_j}}{\frac{1}{\lambda_{2j}}-\frac{1}{\rho_{2j}}}=1-p\Big(1-\frac{1}{q+1}-\frac{2m}{N-1}\Big)-\frac{q}{q+1}\\
&=\frac{pq-1}{(q+1)(N-1)}\big(\alpha-N+2m+1\big)>0.
\end{align*}

{\it Case 2: $\frac{2j+1}{N-1}<\frac{1}{h}<\frac{2m}{N-1}$.} We have $\frac{1}{t_j}=1-\frac{2m-2j-1}{N-1}$.
\begin{align*}
1-b_{2j}&=1-\frac{1}{k}.\frac{\frac{1}{\gamma_{2j}}-\frac{1}{t_j}}{\frac{1}{\gamma_{2j}}-\frac{1}{\delta_{2j}}}-\frac{1}{h}.\frac{\frac{1}{\lambda_{2j}}-\frac{1}{t'_j}}{\frac{1}{\lambda_{2j}}-\frac{1}{\rho_{2j}}}\\
&=1-\frac{1}{h}.\frac{1-\frac{2m}{N-1}}{1-\frac{1}{h}}>1-\frac{1}{h}>0.
\end{align*}

{\it Case 3:  $\frac{2j+1}{N-1}\geq\frac{1}{h}$.} Then we have $\tau_{2j}=1$. Hence
\begin{align*}
1-b_{2j}=1-\frac{1}{k}(1-\nu_{2j})>0.
\end{align*}
Therefore $b_{2j}<1$, for any $j\in \{0,1,..., m-1\}$.
We next verify $a_{2j}(0)>0$. Indeed, 
\begin{align*}
a_{2j}(0)=&-N-2m+\nu_{2j}\left(\alpha+2m +\frac{a+bp}{pq-1}\right)+ \tau_{2j}\left(\beta+2m +\frac{b+aq}{pq-1}\right)\notag\\
&+\left((2m+\alpha)k+\frac{a+bp}{pq-1}k-M \right)\frac{(1-\nu_{2j})}{k} \\
&+\left((2m+\beta)h+\frac{b+aq}{pq-1}h-M\right)\frac{(1-\tau_{2j})}{h}\\
=&-N-2m+2m+\alpha+\frac{a+bp}{pq-1}+2m+\beta+\frac{b+aq}{pq-1}\\
&-M\left(\frac{1-\nu_{2j}}{k}+\frac{1-\tau_{2j}}{h}\right)\\
=& M- Mb_{2j} = M(1-b_{2j})>0.
\end{align*}
From the computations above, we can choose $\varepsilon>0$ small enough such that $\tilde a>0$ and $\tilde b<1$. Theorem is proved.
\smallskip

\section{Proof of Theorem \ref{th2}.}

We first need the following decay estimates for solutions of system (\ref{1}). The proof of Proposition~\ref{th3} is totally similar to that in \cite[Theorem 1.3]{Pha12}.  
\begin{proposition}\label{th3}  Let $p,q\geq 1$, $pq\ne 1$, $a,b \geq 0$.  Assume (\ref{condi1}) and (\ref{condi2}). Then there exists a constant $C=C(N,m,p,q,a,b)>0$ such that any nonnegative solution of system~(\ref{1}) in $\Omega=\{x\in{\mathbb R}^N:\, |x|>\rho\}$ ($\rho\ge 0$) satisfies
\begin{align*}
&u(x)\leq C|x|^{-\alpha -\frac{a+bp}{pq-1}},
\quad v(x)\leq C|x|^{-\beta -\frac{b+aq}{pq-1}}
,
\quad  |x|>2\rho.\\
&|\nabla^i u(x)|\leq C|x|^{-i-\alpha -\frac{a+bp}{pq-1}},
|\nabla^i u(x)|\leq C|x|^{-i-\beta -\frac{b+aq}{pq-1}},
\; |x|>2\rho, \; i=0,...,2m-1.
\end{align*}
\end{proposition}

\medskip
{\it Proof of Theorem \ref{th2}.}
Let $(u,v)$ be a positive solution of system (\ref{1}). 
By the Rellich-Pohozaev identity (Lemma~\ref{pohozaev})  with (\ref{choicea1a2}), we have
\begin{align*}
\int_{B_R}|x|^av^{p+1}\, dx
+\int_{B_R}|x|^bu^{q+1}\, dx  \leq G_1(R)+G_2(R)
\end{align*}
Where $G_1, G_2$ are defined in (\ref{G1}) and (\ref{G2}). Using the decay estimates in Proposition~\ref{th3}, we have, 
$$G_1(R)+G_2(R)\leq CR^{N-2m-(1+\frac{b}{2m})\alpha-(1+\frac{a}{2m})\beta} \to 0, \text{ as } R\to \infty, $$ 
due to (\ref{condi2}) (which implies (\ref{hyperbola})). Therefore, $u\equiv v\equiv 0$.
\qed

\medskip
{\bf Acknowledgement.} The author would like to thank Professor Philippe Souplet for valuable suggestions and comments. %This research is funded by Vietnam National Foundation for Science and Technology Development (NAFOSTED) under grant number 101.02-2014.06.
%\bibliographystyle{plain}
%\bibliography{C:/Users/Administrator/Dropbox/BIBTEX/henon-ref}

\end{document}